\documentclass[a4paper,11pt]{article}
\usepackage{indentfirst,latexsym,bm,amsmath,amssymb,amsthm}
\usepackage[dvips]{graphicx}
\usepackage{xcolor}
\usepackage{natbib}

\makeatletter\@addtoreset{equation}{section} \makeatother
\newtheorem{theorem}{Theorem}[section]
\newtheorem{proposition}[theorem]{Proposition}%[section]
\newtheorem{corollary}[theorem]{Corollary}%[section]
\theoremstyle{remark}
\newtheorem{remark}[theorem]{Remark}%[section]
\theoremstyle{definition}
\newtheorem{definition}[theorem]{Definition}%[section]
\newtheorem{exmp}[theorem]{Example}%[section]

 

\title{Random Locations of Periodic Stationary Processes}   %title

\author{Jie Shen, Yi Shen, Ruodu Wang}
\date{{\small \sc Department of Statistics and Actuarial Science\\
University of Waterloo}\\%~\\\today
}
\providecommand{\keywords}[1]{\textbf{\textit{Keywords: }} #1}
\begin{document}
\maketitle
\begin{abstract}
%%abstract
We consider a family of random locations, called intrinsic location functionals, of periodic stationary processes. This family includes but is not limited to the location of the path supremum and first/last hitting times. We first show that the set of all possible distributions of intrinsic location functionals for periodic stationary processes is the convex hull generated by a specific group of distributions. We then focus on two special subclasses of these random locations. For the first subclass, the density has a uniform lower bound; for the second subclass, the possible distributions are closely related to the concept of joint mixability.
\end{abstract}
\section{\textbf{Introduction}}
Random locations of stationary processes have been studied for a long time, and various results exist for special random locations and processes. For example, the results regarding the hitting time for Ornstein-Uhlenbeck processes date back to Breiman's paper in 1967 \cite{breiman1967}, with recent developments made by Leblanc et al. \cite{leblanc2000correction} and Alili et al. \cite{alili2005representations}. Early discussions about the location of path supremum over an interval can be found in the work of Leadbetter et al \cite{leadbetter1983extremes}. The book by Lindgren \cite{lindgren2012stationary} provides an excellent summary of general results in stationary processes.

Recently, properties of possible distributions of the location of the path supremum have been obtained, and the sufficiency of the properties was proven \cite{samorodnitsky2012distribution, samorodnitsky2013location}. In \cite{samorodnitsky2013intrinsic}, Samorodnitsky and Shen proceeded to introduce a general type of random locations called intrinsic location functionals, including but also extending far beyond the random locations mentioned above. In \cite{shen2016random}, equivalent representations of intrinsic location functionals were established using partially ordered random sets and piecewise linear functions.

In this paper, we study intrinsic location functionals of periodic stationary processes, and characterize all the possible distributions of these random locations. The periodic setting leads to new properties along with challenges, which are the focus of this paper. The periodicity also adds a discrete flavor to the problem, which, surprisingly, suggests a link with other well-studied properties such as joint mixability \citep{wang2016joint}.

The motivation of this work is twofold. From the general theoretical perspective, since the study of continuous-time stationary processes requires a differentiable manifold structure to apply analysis techniques as well as a group structure to define stationarity, the most general and natural framework under which the random locations of stationary processes can be considered is an Abelian Lie group. It is well known that any connected Abelian Lie group can be represented as the product of real lines and one-dimensional torus, \textit{i.e.}, circles. In other words, the real line $\mathbb R$ and one-dimension circle $S_1$ are building blocks for connected Abelian Lie groups. Therefore, in order to understand the properties of  random locations of stationary processes in the general setting, it is crucial to study their behaviors on $\mathbb R$ and $S_1$ first. While the case for $\mathbb R$ was done in \cite{samorodnitsky2013location}, this paper deals with the circular case, which is equivalent to imposing a periodic condition on the stationary processes over the real line.

A more specific motivation comes from a problem in the extension of the so-called ``relatively stationary process''. A relatively stationary process is, briefly speaking, a stochastic process only defined on a compact interval, the finite dimensional distribution of which is invariant under translation, as long as all the time indices in the distribution remain inside the interval. Parthasarathy and Varadhan \cite{parthasarathy1964extension} showed that a relatively stationary process can always be extended to a stationary process over the whole real line. A question to ask as the next step is when such an extension can be periodic. Equivalently, if the relatively stationary process is defined on an arc of a circle instead of the compact interval on the real line, can it always be extended to a stationary process over the circle? This paper will provide an answer to this question.

The rest of the paper is organized as follows. In Section $2$, we introduce some notation and assumptions for intrinsic location functionals and stationary and ergodic processes. In Section $3$, we show some general results on intrinsic location functionals of periodic stationary processes. Sufficient and necessary conditions are established to characterize the distributions of these random locations. The following two sections are devoted to two special types of intrinsic location functionals. In Section $4$, the class of \emph{invariant intrinsic location functionals} is studied. The density of any invariant intrinsic location functional has a uniform lower bound, and such a distribution can always be constructed via the location of the path supremum over the interval. In Section $5$, we show that the density of a \emph{first-time intrinsic location functional} is non-increasing, and establish a link between the structure of the set of first-time intrinsic locations' distributions and the joint mixability of some distributions.

\section{\textbf{Notation and preliminaries}}
Throughout the paper, $\mathbf X=\{X(t),~t\in \mathbb R\}$ will denote a periodic stationary process. Without loss of generality, assume  $\mathbf X$ has period  $1$. Moreover, for simplicity, we assume the sample function $X(t)$ is continuous unless specified otherwise. Indeed, all the arguments in the following parts also work for $\mathbf X$ with c\`{a}dl\`{a}g sample paths.

As mentioned in the Introduction, an equivalent description of a periodic stationary stochastic process is a stationary process on a circle. That is, consider $\{X(t), ~t\in \mathbb R\}$ as a process defined on $S_1$, where $S_1$ is a circle with perimeter $1$.

Let $H$ be a set of functions on $\mathbb{R}$ with period $1$, and assume it is invariant under shifts. The latter means that for all $g\in H$ and $c\in \mathbb {R}$, the function $\theta_{c}g(x):= g(x+c)$, $x\in \mathbb{R}$ belongs to $H$. We equip $H$ with its cylindrical $\sigma$-field. Let $\mathcal{I}$ be the set of all compact, non-degenerate intervals in $\mathbb{R}$: $\mathcal{I}=\{[a,b]: a<b,~ [a,b] \subset \mathbb R\}$.
We first define intrinsic location functionals, the primary object of this paper.

%%definition of intrinsic location functions
\begin{definition}\label{def:ilf}
\citep{samorodnitsky2013intrinsic} A mapping $L$: $H\times \mathcal{I} \to \mathbb{R}\cup \{\infty\}$ is called an \emph{intrinsic location functional}, if it satisfies the following conditions:
\begin{enumerate}
\item For every $I\in \mathcal{I}$, the mapping $L(\cdot,I):H\to \mathbb{R}\cup\{\infty\}$ is measurable.
\item For every $g\in H$ and $I\in \mathcal{I}$, $L(g,I)\in I\cup\{\infty\}$.
\item (Shift compatibility) For every $g\in H$, $I\in \mathcal{I}$ and $c\in \mathbb{R}$,
\begin{align*}
L(g,I)=L(\theta_c g,I-c)+c,
\end{align*}
where $I-c$ is the interval $I$ shifted by $-c$, and by convention, $\infty+c=\infty$.
\item (Stability under restrictions) For every $g\in H$ and $I_1$, $I_2\in \mathcal{I}$, $I_2\subseteq I_1$, if $L(g,I_1)\in I_2$, then $L(g,I_2)=L(g,I_1)$.
\item (Consistency of existence) For every $g\in H$ and $I_1$, $I_2\in \mathcal{I}$, $I_2\subseteq I_1$, if $L(g,I_2)\neq \infty$, then $L(g,I_1)\neq \infty$.
\end{enumerate}
\end{definition}

All the conditions in Definition \ref{def:ilf} being natural and general, the family of intrinsic location functionals is a very large family of random locations, including and extending far beyond the location of the path supremum/infimum, the first/last hitting times, the location of the first/largest jump, \textit{etc}.

\begin{remark}
$\infty$ is added to the range of the intrinsic location functionals to deal with the issue that some intrinsic location functionals may not be well defined for certain paths in some intervals. The $\sigma$-field on $\mathbb R\cup \{\infty\}$ is then given by treating $\{\infty\}$ as a separate point and taking the $\sigma$-field generated by the Borel sets in $\mathbb R$ and $\{\infty\}$.
\end{remark}

It turns out that with the presence of a period, the relation between stationary processes and ergodic processes plays a crucial role in analyzing the distributions of the random locations. Let $(\Omega, \mathcal F, \mathbb P)$ be a probability space. Recall that a measurable function $f$ is called \emph{$T$-invariant} for a measurable mapping $T: \Omega\to\Omega$, if
\begin{align*}
f(T\omega)=f(\omega) \quad {\mathbb P}\text{-almost surely}.
\end{align*}
For a stationary process $\mathbf X=\{X(t),~t\in\mathbb R\}$, let $\tilde{\Omega}$ be its canonical space equipped with the cylindrical $\sigma$-field $\tilde{\mathcal F}$, and $\theta_t$ be the shift operator as defined earlier. That is,
\begin{align*}
\theta_t\tilde{\omega}(s)=\tilde{w}(s+t), \text{ for }\tilde{\omega}\in \tilde{\Omega}.
\end{align*}
Denote by  $\mathbb P_{\mathbf X}(\cdot)=\mathbb P(\mathbf  X\in \cdot)$ the distribution of $\mathbf X$ on $(\tilde{\Omega},\tilde{\mathcal F})$. A stationary process $\{X(t),~t\in \mathbb R\}$ is called ergodic, if each measurable function $f$ defined on $(\tilde{\Omega},\tilde{\mathcal F})$ which is $\theta_t$-invariant for every $t$ is constant $\mathbb P_\mathbf{X}$-almost surely.

It is known that the set of the laws of all stationary processes is a convex set and the extreme points of this set are the laws of the ergodic processes. Thus, we have the ergodic decomposition for stationary processes:
\begin{theorem}\label{thm5}
(Theorem A.1.1, Kifer \citep{kifer1988ergodic}) Let $\mathcal M$ be the space of all stationary probability measures, and $\mathcal M_e$ the subset of $\mathcal M$ consisting of all ergodic probability measures. Equip $\mathcal M$ and $\mathcal M_e$ with the natural $\sigma$-field: $\sigma(\mu\to \mu(A): A\in \mathcal F)$. For any stationary probability measure $\mu_\mathbf{X}\in \mathcal M$, there exists a probability measure $\lambda$ on $\mathcal M_e$ such that
\begin{align*}
\mu_{\mathbf{X}}=\int_{\rho \in \mathcal M_e}\rho \mathrm{d}\lambda.
\end{align*}
\end{theorem}

The following proposition shows that for periodic stationary processes, ergodicity simply means that all the paths are the same up to translation. This simple fact will be used later in showing the main results of this paper.

We say a probability space $(\Omega, \mathcal F, \mathbb P)$ can be extended to a probability space $(\tilde{\Omega}, \tilde{\mathcal F}, \tilde{\mathbb P})$, if there exists a measurable mapping $\pi$ from $(\tilde{\Omega}, \tilde{\mathcal F})$ to $(\Omega, \mathcal F)$ satisfying $\mathbb P=\tilde{\mathbb P}\circ \pi^{-1}$. In this case, the process $\tilde{\mathbf X}$ defined on $(\tilde{\Omega}, \tilde{\mathcal F}, \tilde{\mathbb P})$ by $\tilde{\mathbf X}(\tilde \omega)=\mathbf X(\pi(\tilde{\omega}))$ will be identified with the original process $\mathbf X$.
%In the following proposition, we regard a stochastic process $\mathbf X$ as originally defined as the coordinate process on its canonical space. Note that if $\mathbf X$ is continuous and periodic with period 1, then its paths can be identified with functions in $C([0,1])$ having the same value at $0$ and $1$, and that the Borel $\sigma-$field and the cylindrical $\sigma-$field coincide on $C([0,1])$.
\begin{proposition}\label{cha}
For any continuous periodic ergodic process $\mathbf{X}$ with period $1$, there exists a deterministic function $g$ with period $1$, such that $X(t)=g(t+\tilde U)$ for $t\in \mathbb R$ almost surely on an extended probability space, in which $\tilde U$ follows a uniform distribution on $[0, 1]$.
\end{proposition}
\begin{proof}
Let $C_1(\mathbb R)$ be the space of continuous functions with period 1. For $h\geq 0$, define set $B_h:=\{g\in C_1(\mathbb R): \sup_{t\in \mathbb R}|g(t)|\leq h\}$. Note that $B_h$ is in the invariant $\sigma$-algebra, and hence by ergodicity, $\mathbb P(\mathbf X\in B_h)$ is either $0$ or $1$ for any $h$. Consequently, there exists $h_0$ (depending on $\mathbf X$) such that $\mathbb P(\mathbf X\in B_{h_0})=1$.

Similarly, for function $\delta: [0,\infty)\to[0,\infty)$, define set
$$
C_\delta:=\{g\in C_1(\mathbb R): |g(x)-g(y)|<\varepsilon \text{ for any } \varepsilon>0 \text{ and all } |x-y|<\delta(\varepsilon)\},
$$
then $C_\delta$ is in the invariant $\sigma$-algebra, $\mathbb P(\mathbf X\in C_\delta)\in\{0,1\}$, and there exists function $\delta_0$ such that $\mathbb P(\mathbf X\in C_{\delta_0})=1$.

Furthermore, for any $n$, $\mathbf t=(t_1,...,t_n)$ and $\mathbf A=(A_1,...,A_n)$, where $t_1<t_2<\dots<t_n$ and $A_1,...,A_n$ are non-degenerate closed intervals, define sets
$$
H_{\mathbf t, \mathbf A}:=\{g\in C_1(\mathbb R): g(t_1)\in A_1,\dots, g(t_n)\in A_n\}
$$
and
$$
H_{\mathbf t, \mathbf A}^0:=\{g\in C(\mathbb R): \text{ there exists a} \text{ constant }c, ~\theta_c g\in H_{\mathbf t, \mathbf A}\}.
$$
Again, $H_{\mathbf t, \mathbf A}^0$ is in the invariant $\sigma$-algebra, and hence by ergodicity $\mathbb P(\mathbf X\in H_{\mathbf t, \mathbf A}^0)$ is either $0$ or $1$ for any $n$, $t_1,...,t_n$ and $A_1,...,A_n$.

For $m=0,1,...$, let $n_m=2^m$ and $t^m_i=(i-1)2^{-m}$ for $i=1,...,n_m$. Then there exists $A^m_1,...,A^m_{n_m}$ of the form $A^m_i=[k_i2^{-m}, (k_i+1)2^{-m}]$, $k_i\in \mathbb Z, i=1,...,n_m$, such that $\mathbb P(\mathbf X \in H_{\mathbf t^m, \mathbf A^m}^0)=1$, where $\mathbf t^m=(t^m_1,...,t^m_{n_m})$, $\mathbf A^m=(A^m_1,...,A^m_{n_m})$. Moreover, we can choose the sets such that $\{H_{\mathbf t^m, \mathbf A^m}^0\}_{m=0,1,...}$ form a decreasing sequence, \textit{i.e.}, $H_{\mathbf t^{m_1}, \mathbf A^{m_1}}^0\supseteq H_{\mathbf t^{m_2}, \mathbf A^{m_2}}^0$ if $m_1\leq m_2$.

Consider the sequence of sets $\{H^0_{\mathbf t^m, \mathbf A^m}\cap B_{h_0}\cap C_{\delta_0}\}_{m=0,1,...}$. Each set in this sequence is closed and consists of functions which are uniformly bounded and equicontinuous. By Arzel\`{a}-Ascoli Theorem and the fact that we are looking at functions with period 1, which can be 1-1 mapped to $\{g\in C([0,1]): g(0)=g(1)\}\subset C([0,1])$, the sets in this sequence are compact. As a result, the intersection of all the sets is non-empty. Moreover, there exists a single deterministic function with period 1, denoted by $g$, such that for any $f$ in the intersection, $f(t)=g(t+c)$ for some $c\in \mathbb R$. Indeed, assume this is not the case, i.e., there exists $f_1, f_2$ both in $H^0_{\mathbf t^m, \mathbf A^m}\cap B_{h_0}\cap C_{\delta_0}$ for all $m=0,1,...$, yet $f_1\neq \theta_cf_2$ for any $c$, then fundamental analysis shows that
$$
\inf_{c\in\mathbb R}\sup_{i\in \mathbb Z}|f_1(i2^{-m})-\theta_cf_2(i2^{-m})|\geq \frac{1}{2}\inf_{c\in\mathbb R}\sup_{t\in \mathbb R}|f_1(t)-\theta_cf_2(t)|>0
$$
for $m$ large enough, hence $f_1$ and $f_2$ will eventually be separated by some $H^0_{\mathbf t^m, \mathbf A^m}$. Thus, we conclude that $X(t)=g(t+V)$ almost surely for some random variable $V$.

The last step is to show that there exists an extended probability space and a uniform $[0,1]$ random variable $\tilde U$ defined on that space, such that $X(t)=g(t+\tilde {U})$ almost surely. First, suppose there exists a uniform [0,1] random variable $U$ in some probability space, then $\{X(t), {t\in\mathbb R}\}\stackrel{d}{=}\{g(t+U), t\in\mathbb R\}$. Indeed, since the equality is in the distributional sense, we can assume that $U$ is independent of everything else by considering, for example, the product space of the original probability space and $[0,1]$ quipped with the Borel $\sigma-$field and the Lebesgue measure. Then by stationarity and ergodicity, we have
\begin{align*}
\{X(t), t\in\mathbb R\}& \stackrel{d}{=}\{X(t+U), t\in\mathbb R\}\\
& = \{g(t+V+U), t\in\mathbb R\}\\
& \stackrel{d}{=} \{g(t+U), t\in\mathbb R\}.
\end{align*}
Moreover, the mapping $h:[0,1]\to C([0,1])$ given by $h(x)=\{g(t+x), t\in[0,1]\}$ is continuous, hence measurable. (Note that the Borel $\sigma-$field and the cylindrical $\sigma-$field coincide on $C([0,1])$.) As a result, there exists an extended probability space $(\tilde{\Omega}, \tilde{\mathcal F}, \tilde{\mathbb P})$ with a uniform [0,1] random variable $\tilde U$ defined on that, such that $\{X(t), t\in \mathbb R\}=h(\tilde{U})=\{g(t+\tilde U), t\in\mathbb R\}$ almost surely on $(\tilde{\Omega}, \tilde{\mathcal F}, \tilde{\mathbb P})$.

\end{proof}

\section{\textbf{Distributions of intrinsic location functionals}}

In this section, we characterize (properties of) intrinsic location functionals of periodic stationary processes.
For a compact interval $[a,b]$, denote the value of an intrinsic location functional $L$ for the process $\mathbf{X}$ on that interval by $L(\mathbf{X}, [a,b])$. Since $\mathbf X$ is stationary and $L$ is shift compatible, the distribution of $L-a$ depends solely on the length of the interval. Thus, we can focus on the intervals starting from $0$, in which case $L(\mathbf{X}, [0,b])$ is abbreviated as $L(\mathbf{X},b)$. Furthermore, with the 1-periodicity of $\mathbf{X}$, it turns out that the only interesting cases are those with $b\leq 1$. In the following we assume $b\leq 1$ throughout. The case where $b>1$ will be briefly discussed in Remark \ref{remark:longer}, after the introduction of a representation result for intrinsic location functional.

Denote by $F^\mathbf{X}_{L, [a,b]}$ the law of $L(\mathbf{X}, [a,b])$. It is a probability measure supported on $[a,b]\cup \{\infty\}$.

It was shown in \cite{samorodnitsky2013intrinsic} that the distribution of an intrinsic location functional for any stationary process over the real line, not necessarily periodic, possesses a specific group of properties. Adding periodicity obviously will not change these results. Here we present a simplified version of the original theorem for succinctness.

\begin{proposition}\label{prop2}
Let $L$ be an intrinsic location functional and $\{X(t), ~t\in \mathbb R\}$ a stationary process. The restriction of the law $F^\mathbf{X}_{L, T}$ to the interior $(0,T)$ of the interval is absolutely continuous. Moreover, there exists a c\`adl\`ag version of the density function, denoted by $f^\mathbf {X}_{L, T}$, which satisfies the following conditions:\\
(a) The limits
\begin{align}
f^\mathbf {X}_{L, T}(0+)=\lim_{t \downarrow 0}f^\mathbf {X}_{L, T}(t) ~\mbox{and}~ f^\mathbf {X}_{L, T}(T-)=\lim_{t\uparrow T}f^\mathbf {X}_{L, T}(t)
\end{align}
exist.\\
(b) \begin{align*}
\mathrm{TV}_{(t_1,t_2)}(f^\mathbf {X}_{L,T})\leq f^\mathbf {X}_{L, T}(t_1)+f^\mathbf {X}_{L, T}(t_2)
\end{align*}
for all $0<t_1<t_2<T$, where
\begin{align*}
\mathrm{TV}_{(t_1,t_2)}(f^\mathbf {X}_{L, T})=\sup \sum^{n-1}_{i=1} \left |f^\mathbf {X}_{L, T}(s_{i+1})-f^\mathbf {X}_{L, T}(s_i)\right |
\end{align*}
is the total variation of $f^\mathbf {X}_{L, T}$ on the interval $(t_1,t_2)$, and the supremum is taken over all choices of $t_1<s_1<\cdots <s_n<t_2$.\\
\end{proposition}
Note that we have $\int^T_0 f^\mathbf {X}_{L,T}(s) \mathrm {d}s<1$ if there exists a point mass at $\infty$ or at the boundaries $0$ and $T$.

We call the condition $(b)$ in Proposition \ref{prop2} ``Condition $(TV)$'', or the ``variation constraint'', because it puts a constraint on the total variation of the density function. It is not difficult to show that Condition $(TV)$ is equivalent to the following Condition $(TV^{\prime})$:

There exists a sequence $\{t_n\}$, $t_n\downarrow 0$, such that
\begin{align*}
\mathrm{TV}_{(t_n,T-t_n)}(f)\leq f(t_n)+f(T-t_n),~n\in \mathbb N .
\end{align*}

The above general result about the distribution of the intrinsic location functionals for stationary processes over the real line is still valid for periodic stationary processes, and serves as a basis for further exploration. It is, however, not the focus of this paper. For the rest of the paper we will concentrate on the new properties introduced by the periodicity assumption, which do not hold in the general case.

For any intrinsic location functional $L$ and $T\leq 1$, let $I_{L,T}$ be the set of probability distributions $F_{L,T}^{\mathbf X}$ for periodic stationary processes $\mathbf X$ with period $1$ on $[0,T]$. Our goal is to understand the structure of the set $I_{L,T}$, and the conditions that the distributions in $I_{L,T}$ need to satisfy. To this end, note that since ergodic processes are extreme points of the set of stationary processes, the extreme points of the set $I_{L,T}$ can only be the distributions of $L$ for periodic ergodic processes with period $1$. The next proposition gives a list of properties for these distributions.

\begin{proposition}\label{extreme}
Let $L$ be an intrinsic location functional, $\mathbf X$ be a periodic ergodic process with period 1, and $T\leq 1$. Then $F_{L,T}^{\mathbf X}$ and its c\`{a}dl\`{a}g density function on $(0,T)$, denoted by $f$, satisfy:
\begin{enumerate}
\item $f$ takes values in non-negative integers;
\item $f$ satisfies the condition $(TV)$;
\item If $F_{L,T}^{\mathbf X}[0,T]>0$, and there does not exist $t\in(0,T)$ such that $F_{L,T}^{\mathbf X}[0,t]=1$ or $F_{L,T}^{\mathbf X}[t,T]=1$, then $f(t)\geq 1$ for all $t\in (0,T)$. If furthermore, $F_{L,T}^{\mathbf X}(\{\infty\})>0$, then $f-1$ also satisfies the condition $(TV)$.
\end{enumerate}
\end{proposition}

Note that the condition in the first part of property 3 can be translated into requiring either a positive but smaller than 1 mass at $\infty$, or a positive point mass or a positive limit of the density function at each of the two boundaries 0 and $T$.

The proof of Proposition \ref{extreme} relies on the following representation result given in \cite{shen2016random}.

\begin{proposition}\label{prop:re}
A mapping $L(g,I):H\times \mathcal I\to \mathbb R\cup \{\infty\}$ is an intrinsic location functional if and only if
\begin{enumerate}
\item $L(\cdot, I)$ is measurable for $I\in \mathcal I$;
\item There exists a subset of $\mathbb R$ determined by $g$, denoted as $S(g)$, and a partial order $\preceq$ on it, satisfying:
\begin{enumerate}
\item[(1)] For any $c\in \mathbb R$, $S(g)=S(\theta_c g)+c$;
\item[(2)] For any $c\in \mathbb R$ and $t_1,t_2\in S(g)$, $t_1\preceq t_2$ implies $t_1-c\preceq t_2-c$ in $S(\theta_c g)$,
\end{enumerate}
\end{enumerate}
such that for any $I\in \mathcal I$, either $S(g)\cap I=\emptyset$, in which case $L(g,I)=\infty$, or $L(g,I)$ is the unique maximal element in $S(g)\cap I$ according to $\preceq$.
\end{proposition}

Such a pair $(S, \preceq)$ in the above proposition is called a \textit{partially ordered random set representation of $L$}. Intuitively, this representation result shows that a random location is an intrinsic location functional if and only if it always takes the location of the maximal element in a random set of points, according to some partial order. Both the random set and the order are determined by the path and are shift-invariant.

\begin{remark}\label{remark:longer}
By Proposition \ref{prop:re}, for a function $g$ with period 1, $t\in S(g)$ implies $t+c\in S(\theta_{-c} g)=S(g)$ for any $c\in\mathbb Z$. Moreover, if $t+1\preceq t$, then $t+c_2\preceq t+c_1$ for all $c_1,c_2\in\mathbb Z, c_2>c_1$. As a result, for an interval $[a,b]$ with length greater than 1, only the points in the leftmost cycle $[a,a+1)$ can have the maximal order. Thus, the location of the intrinsic location functional on $[a,b]$ will be the same as on $[a,a+1]$. Symmetrically, if $t\preceq t+1$, then the location of the intrinsic location functional on $[a,b]$ will be the same as on $[b-1,b]$. Hence we only need to consider the intervals with length no larger than $1$.
\end{remark}

\begin{proof}[Proof of Proposition \ref{extreme}]
Property 2 directly comes from Proposition \ref{prop2}. We only need to check properties 1 and 3.

\underline{Property 1}. Since $\mathbf{X}$ is a periodic ergodic process with period $1$, by Proposition \ref{cha}, there exists a periodic deterministic function $g$ with period $1$ such that $X(t)=g(t+U)$ for $t\in \mathbb R$, where $U$ follows a uniform distribution on $[0,1]$. In other words, all the sample paths of $\mathbf{X}$ are the same up to translation. Let $(S,\preceq)$ be a partially ordered random set representation of $L$. For any $s\in S(g)$, define
\begin{align*}
a_s&:=\sup \{\Delta s\in \mathbb{R}: r\preceq s~\mbox{for all}~ r\in (s-\Delta s,s)\cap S(g)  \},\\
b_s&:= \sup \{\Delta s\in \mathbb{R}: r\preceq s~\mbox{for all}~ r\in(s,s+\Delta s)\cap S(g) \},
\end{align*}
and define $\sup \emptyset=\infty$ by convention. By a slight abuse of notation, we also use $a_s$ and $b_s$ to denote the same quantity for $s\in S(\mathbf X)$. Intuitively, $a_s$ and $b_s$ are the largest distance by which we can go to the left and right of the point $s$ without passing a point with higher order than $s$ according to $\preceq$, respectively. Thus, for $0<t<t+\Delta t<T$, we have
\begin{align}
\nonumber
&\quad \mathbb P\left( \text{there exists}\,s\in [t,t+\Delta t]\cap S(\mathbf{X}): a_s> t+\Delta t,~b_s> T-t \right) \\
\nonumber
&\leq \mathbb P\left(t\leq L(\mathbf{X},(0,T))\leq t+\Delta t \right)\\
&\leq \mathbb P\left( \text{there exists}\,s\in [t,t+\Delta t]\cap S(\mathbf{X}): a_s\geq t,~b_s\geq T-t-\Delta t  \right).
\end{align}
Seeing that $X(t)=g(t+U)$, $S(\mathbf{X})=S(g)-U$. By change of variable $s\to s-U$,
\begin{align*}
& P\left(\mbox{there exists}\,s\in [t,t+\Delta t]\cap S(\mathbf{X}): a_s> t+\Delta t, ~b_s> T-t  \right)\\
= &  P\left(\mbox{there exists}\,s\in S(g): a_s> t+\Delta t, ~b_s> T-t, ~s-U\in [t,t+\Delta t]\right).
\end{align*}
Note the values of $a_s$ and $b_s$ remain unchanged, since they are defined with respect to $\mathbf X$ on the left hand side, and with respect to $g$ on the right hand side.

Since $S(g)$ has period 1, $s\in S(g)$ if and only if $s-\lfloor s\rfloor\in S(g)\cap[0,1)$. Moreover, since $s-U$ and $s-\lfloor s\rfloor-U-\lfloor s-\lfloor s\rfloor-U\rfloor$ share the same fractional part and are both in $[0,1)$, $s-U=s-\lfloor s\rfloor-U-\lfloor s-\lfloor s\rfloor-U\rfloor$. Thus, by another change of variable $s-\lfloor s\rfloor \to s$, we have
\begin{align*}
&\quad  P\left(\mbox{there exists}\,s\in S(g): a_s> t+\Delta t, ~b_s> T-t, ~s-U\in [t,t+\Delta t]\right)\\
&=P\left(\mbox{there exists}\, s\in S(g)\cap [0,1) \,\right.\\
&~~~\left.\mbox{such that}\, a_s>t+\Delta t, ~b_s>T-t, ~\mbox{and}~s-U-\lfloor s-U \rfloor\in[t,t+\Delta t]\right).
\end{align*}

Therefore, for $\Delta t$ small enough,
\begin{align*}
&\quad \mathbb P \left(\text{there exists}\,s\in [t,t+\Delta t]\cap S(\mathbf{X}): a_s> t+\Delta t,~b_s> T-t  \right)\\
&=\left|\{s\in S(g)\cap [0,1): a_s> t+\Delta t,~b_s> T-t\}\right | \cdot \Delta t,
\end{align*}
where $|A|$  denotes the cardinal of set $A$.
Thus, we have
\begin{align}
f(t)&=\lim_{\Delta t\to 0}\frac{ \mathbb P\left(t\leq L(\mathbf{X},(0,T))\leq t+\Delta t \right)}{\Delta t}\nonumber\\
&\geq \left| \{s\in S(g)\cap [0,1): a_s> t, ~b_s> T-t\}\right|.
\end{align}
Symmetrically,
\begin{align}\label{char3}
f(t)&=\lim_{\Delta t\to 0}\frac{ \mathbb P\left(t\leq L(\mathbf{X},(0,T))\leq t+\Delta t \right)}{\Delta t}\nonumber\\
&\leq |\{s\in S(g)\cap [0,1): a_s\geq t,~ b_s\geq T-t\}|.
\end{align}
Moreover, it is easy to see that the set $\Sigma:=\{s \in S(g)\cap [0,1): a_s>0 ~\mbox{and}~ b_s>0\}$ is at most countable, then $\{t: a_s=t ~\mbox{or}~ b_s=T-t~ \mbox{for some}~ s\in \Sigma\}$ is also at most countable. Hence the density can be taken as the c\`{a}dl\`{a}g modification of $\left|\{s\in S(g)\cap [0,1): a_s\geq t, ~b_s\geq T-t\}\right|$, which only takes values in non-negative integers.

\underline{Property 3}. Assume $F_{L,T}^{\mathbf X}[0,T]>0$ and there does not exist $t\in(0,T)$, such that $F_{L,T}^{\mathbf X}[0,t]=1$ or $F_{L,T}^{\mathbf X}[t,T]=1$. There are two possible cases depending on whether $F_{L,T}^{\mathbf X}$ has a point mass at $\infty$.

First suppose $F_{L,T}^{\mathbf X}(\{\infty\})\in (0,1)$. Then by the partially ordered random set representation, there exists an interval $[s_\infty, t_{\infty}]$ (depending on $g$) satisfying $t_\infty-s_\infty\geq T$, such that $S(g)\cap [s_\infty, t_\infty]=\emptyset$. Since $g$ has period 1, $S(g)\cap [s_\infty+1, t_\infty+1]=\emptyset$ as well. Let $\tau=L(g,[t_\infty, s_\infty+1])$. Since $L$ is not identically $\infty$, such a finite $\tau$ must exist. Moreover note that there is no point of $S(g)$ in $[s_\infty, t_\infty]$ and $[s_\infty+1, t_\infty+1]$, hence $\tau$ is actually the maximal element in $S(g)$ according to $\preceq$ on the interval $[s_\infty, t_\infty+1]$. Thus, $a_\tau>\tau-s_\infty=\tau-t_\infty+t_\infty-s_\infty\geq T$, and symmetrically $b_\tau\geq T$. Consequently, $\tau-\lfloor \tau\rfloor$ is in the set $\{s\in S(g)\cap [0,1): a_s\geq t, ~b_s\geq T-t\}$ for all $t\in (0,T)$. Since the density function $f(t)$ can be taken as the c\`{a}dl\`{a}g modification of $\left|\{s\in S(g)\cap [0,1): a_s\geq t, ~b_s\geq T-t\}\right|$, $f(t)\geq 1$ for all $t\in(0,T)$.

For the second possibility, suppose now there is either a positive mass or a positive limit of the density function on each of the two boundaries 0 and $T$. Suppose for the purpose of contradiction that there exists a non-degenerate interval $[u,T-v]$ such that $f(t)=0$ for all $t\in [u,T-v]$. For $t\in S(g)$, we distinguish four different types:
$A:=\{t\in S(g): a_t\leq u,~b_t> T-u-\epsilon \}$, $B:=\{t\in S(g):a_t>T-v-\epsilon,~b_t\leq v\}$, $C:=\{t\in S(g): a_t> u,~b_t > v,~ a_t+ b_t> T\}$ and $D:=\{t\in S(g): a_t>u, ~b_t>v, ~a_t+b_t=T\}$, where $0<\epsilon<\frac{T-u-v}{2}$. Sets $A$, $B$, $C$ and $D$ are disjoint, and for any $t\in S(g)$ such that $t=L(g,I)$ for some interval $I$ with length $T$, $t\in A\cup B\cup C\cup D$. By the assumption about $f$, it is easy to see that $A\neq \emptyset$, $B\neq \emptyset$ and $C=\emptyset$.

We claim that for any $x\in A$ and $y\in B$, if $x>y$, then $x-y > T$. Suppose it is not true. For interval $I=[t,t+T]$, where $t$ satisfies $0\leq y-t<T-v-\epsilon$ and $0\leq t+T-x<T-u-\epsilon$, let $z$ be the maximal element in $S(g)\cap I$ according to $\preceq$. Note that the choice of $t$ guarantees that $x,y\in I$, hence $S(g)\cap I\neq\emptyset$, $z$ always exists. Moreover, $x\preceq z$ and $y\preceq z$. Because $y\in B$, $y$ is larger in $\preceq$ than any point to its left within a distance smaller than $T-v-\epsilon$, which contains $[t,y]$. Thus, $z$ can not be in this part of the interval $I$. Similarly, $z$ can not be in $[x,t+T]$, hence $z\in [y,x]$. For such $z$,
\begin{align*}
a_z\geq a_y>T-v-\epsilon>u, ~~b_z\geq b_x> T-u-\epsilon>v,
\end{align*}
and $a_z+b_z>T-v-\epsilon+T-u-\epsilon>T$, which means $z\in C$. However, $C=\emptyset$ by assumption. Therefore, for any $x\in A$, $y\in B$ and $x>y$, we have $x-y>T$.

On the other hand, we show in the following paragraphs that for any point $y\in B$, there exists another point $y'\in B$, such that $\frac{u}{2}<y'-y\leq T$. To this end, consider a number of intervals $[y-\epsilon_i,y-\epsilon_i+T]$ given any arbitrary point $y\in B$ and $\epsilon_i=\frac{1}{2i}u$ for $i=1,2,\dots$. Denote $l_i$ as the maximal element in $[y-\epsilon_i,y-\epsilon_i+T] \cap S(g)$ according to $\preceq$. Notice that since $y\in S(g)$, $l_i$ always exists. Seeing that $a_y>T-v-\epsilon>u$, $l_i$ must be in $[y,y+T]$. Since $l_i-y\leq T$, $l_i$ must be in the set $B\cup D$.

Next, we show that there exists $i$ such that $l_i\in B$. Suppose $l_i\in D$ for all $i$. If there exist $l_i=l_j\in D$ for some $i< j$, then $l_i$ is the maximal element in both $[y-\epsilon_i, y-\epsilon_i+T]\cap S(g)$ and $[y-\epsilon_j, y-\epsilon_j+T]\cap S(g)$. As a result, we have $a_{l_i}\geq l_i-y+\epsilon_i$, and $b_{l_i}\geq y-\epsilon_j+T-l_i$. However, this leads to
$$
a_{l_i}+b_{l_i}\geq T+\epsilon_i-\epsilon_j>T,
$$
hence $l_i$ can not be in $D$. Thus, for any $i\neq j$, $l_i\neq l_j$. By the fact that $a_{l_i}>u$ and $b_{l_i}>v$, there are at most $\frac{T}{\min \{u,v\}}$ points in the set $D\cap [y,y+T]$, which contradicts the assumption that $l_i\in D\cap [y,y+T]$ for all $i=1,2,\dots$. As a result, there always exists at least one point $l_i\in B$.

Furthermore, for such $l_i$, if $l_i-y\leq \frac{u}{2}$, then
\begin{align*}
b_{l_i}\geq T-\frac{u}{2}-\epsilon_i\geq T-u>v,
\end{align*}
which contradicts the fact that $l_i\in B$. Therefore for any $y\in B$, there always exists a point $y'=l_i  \in B$, such that
\begin{align*}
\frac{u}{2}<y'-y\leq T.
\end{align*}

As a result, for any periodic function $g$ with period 1, there exists $y_1\in B$ and then a sequence of points $\{y_i,~i=2, \dots,k\}$ in $B$ such that for $i=1,\dots,k-1$,
\begin{align*}
\frac{u}{2}< y_{i+1}-y_i\leq T,
\end{align*}
and $k$ is chosen such that
\begin{align*}
y_{k-1}<1+y_1\leq y_k.
\end{align*}

 However, since $g$ is a periodic function with period $1$ and $A\neq \emptyset$, this means that there must exist some points $x\in A$ and $y\in B$ such that $x-y\leq T$, which contradicts the result we derived before. Therefore, we conclude that there does not exist a non-degenerate interval $[u,T-v]$ such that $f(t)=0$ for all $t\in [u,T-v]$, if the condition in the first part of property 3 holds.

Finally we turn to the second part in property 3. Assume $F_{L,T}^{\mathbf X}(\{\infty\})>0$, then we show that $f-1$ will satisfy the condition (TV). Recall that a positive probability at $\infty$ for $F_{L,T}^{\mathbf X}$ implies the existence of a maximal interval $[s_\infty, t_\infty]$ depending on $g$ satisfying $t_\infty-s_\infty\geq T$ and $S(g)\cap[s_\infty, t_\infty]=\emptyset$. Indeed, the inequality $t_\infty-s_\infty\geq T$ can be strengthened to $t_\infty-s_\infty> T$, since otherwise its contribution to the point mass at $\infty$ will be 0, even though it allows one particular value of $U$ such that $g(t+U)\cap [0,T]=\emptyset$. Consider an interval $[u,v]\subset (0,T)$, such that $f$ is flat on $[u,v]$. Since $f$ takes integer values and satisfies the variation constraint, such an interval always exists. Define
$$
S'(g)=S(g)\cup\{s_\infty+v-\epsilon+C: C\in\mathbb Z\}\cup\bigcup_{C\in \mathbb Z}(s_\infty+T+\epsilon+C, t_\infty+C)
$$
  for $\epsilon$ small enough, and extend the order $\preceq$ to $S'(g)$ (still denoted by $\preceq$) by setting $s_\infty+v-\epsilon+C\preceq t_1\preceq t_2\preceq t$ for any $C\in\mathbb Z$, $t_1,t_2\in (s_\infty+T+\epsilon+C, t_\infty+C)$, $t_1<t_2$, and any $t\in S(g)$. Intuitively, the extended order assigns the minimal order to $s_\infty+v-\epsilon$, then an increasing order to the points in $(s_\infty+T+\epsilon, t_\infty)$, while keeping the order for the added points always inferior to the original points in $S(g)$, and is finally completed by a periodic extension to $\mathbb R$. Let $L'$ be an intrinsic location functional having $(S'(g), \preceq)$ as its partially ordered random set representation, and denote by $f'$ the density of $F_{L',T}^\mathbf X$. It is easy to see that $f'=f+{\mathbb I}_{(v-2\epsilon, v-\epsilon]}$. Hence for $\epsilon$ small enough and $t_n\downarrow 0$ with $t_1$ being small enough, $\text{TV}_{(t_n,T-t_n)}(f')=\text{TV}_{(t_n,T-t_n)}(f)+2$ for any $n$. Since $f'$ satisfies the condition $(TV)$, we must have $\text{TV}_{(t_n,T-t_n)}(f)+2\leq f(t_n)+f(T-t_n)$. Thus $\text{TV}_{(t_n,T-t_n)}(f-1)\leq (f(t_n)-1)+(f(T-t_n)-1)$, which is the variation constraint for $f-1$.
\end{proof}

With the properties of the distributions of $L$ for periodic ergodic processes with period 1 at hand, we proceed to study the structure of $I_{L,T}$, the set of all distributions of $L$ for periodic stationary processes. Denote by $E_{T}$ the collection of probability distributions on $[0,T]\cup\{\infty\}$ satisfying the three properties listed in Proposition \ref{extreme}, and let $\mathcal P_T$ be the collection of all probability distributions on $[0,T]\cup\{\infty\}$ which are absolutely continuous on $(0,T)$. For the rest of the paper, denote by $C(A)$ the convex hull generated by a set $A\subseteq \mathcal P_T$ under the weak topology.

\begin{theorem}\label{convexh}
$I_{L,T}$ is a convex subset of $\mathcal P_T$. Moreover, $I_{L,T}\subseteq C(E_{T})$.
\end{theorem}

\begin{proof}
The convexity of $I_{L,T}$ is obvious. If $F_1, F_2\in I_{L,T}$, then there exist stationary processes with period 1, denoted by $\mathbf X_1, \mathbf X_2$, such that $F_1=F_{L,T}^{\mathbf X_1}$ and $F_2=F_{L,T}^{\mathbf X_2}$. For any $a\in[0,1]$, $aF_1+(1-a)F_2=F_{L,T}^{\mathbf X}$, where the process $\mathbf X$ is a mixture of $\mathbf X_1$ and $\mathbf X_2$, with weights $a$ and $1-a$, respectively.

Next we show $I_{L,T}\subseteq C(E_{T})$. By ergodic decomposition, any $F\in I_{L,T}$ can be written as $F=\int_{G\in E_T}G \mathrm{d} \lambda$, where $\lambda$ is a probability measure on $E_T$. The integration holds in the sense of mixture of probability measures, \text{i.e.},
$$
\int\limits_{x\in[0,T]\cup\{\infty\}}h(x)\mathrm{d} F(x)=\int\limits_{G\in E_T}\int\limits_{x\in[0,T]\cup\{\infty\}}h(x)\mathrm{d} G(x)\mathrm{d} \lambda
$$
for all bounded and continuous function $h$ defined on $[0,T]\cup\{\infty\}$. Since the set of probability measures on $[0,T]\cup\{\infty\}$ equipped with the weak topology is separable, we conclude that $F\in C(E_{T})$.
\end{proof}

The converse of Theorem \ref{convexh}, that for an arbitrarily given intrinsic location functional $L$ and any distribution $F\in C(E_T)$ there exists a periodic stationary process $\mathbf X$ such that $F=F_{L,T}^{\mathbf X}$, is not true in general. For example, it can be easily checked that $L(g,I=[a,b]):=a$ is an intrinsic location functional. Yet the only possible distribution for $L$ on $[0,T]$ is a Dirac measure on the boundary 0. However, the next result shows that the converse does hold if we do not focus on any particular $L$, but collect the possible distributions for all the intrinsic locations functionals. In other words, any member in $C(E_T)$ can be the distribution of some intrinsic location functional on $[0,T]$ and some periodic stationary process with period $1$. More formally, define $I_T=\bigcup_L I_{L,T}$ to be the set of all possible distributions of intrinsic location functionals on $[0,T]$, then $I_T=C(E_T)$. Here and throughout the paper, when we discuss the existence of a stochastic process without specifying the underlying probability space, the existence should be understood as that of the process together with the existence of a probability space on which the process is defined.

\begin{theorem}\label{chage}
For any $F\in C(E_T)$, there exist an intrinsic location functional and a periodic stationary process with period $1$, such that $F$ is the distribution of this intrinsic location for such process on $[0,T]$.
\end{theorem}

The proof of Theorem \ref{chage} consists of three parts. The main steps of the proof are presented in Part I below.
Parts II and III are put in Sections \ref{sec:4} and \ref{sec:5}, respectively, due to the explicit construction required for specific types of intrinsic location functionals.
\begin{proof}[Proof of Theorem \ref{chage}, Part I]
We define an intrinsic location functional $L=L(g,I)$ as
\begin{align*}
L(g,I)=\begin{cases}
L_{1}(g,I) &\mbox{if } g(t)\geq 0 \text{ for all } t\in \mathbb R,\\
L_{2}(g,I) &\mbox{if } \text{there exists } t\in \mathbb R \mbox{ such that } g(t)=-1,\\
L_{3}(g,I) &\mbox{otherwise},
\end{cases}
\end{align*}
where
$$
L_{1}(g,I)=\inf \left\{ t\in I: g(t)=\sup_{s\in I} g(s), g(t)\geq \frac{1}{2} \right\},
$$
$$
L_{2}(g,I)=\inf\{t\in I: g(t)=-1\},
$$
and
$$
L_{3}(g,I)=\sup\{t\in I: g(t)=-2\}.
$$
Intuitively, $L_{1}$ is based on the location of the path supremum, but truncated at level $\frac{1}{2}$. $L_{2}$ and $L_{3}$ are first and last hitting times, respectively.

We first show that such $L$ is an intrinsic location functional, by using the partially ordered random set representation of intrinsic location functionals. It is not difficult to verify that $L_{1}$, $L_{2}$ and $L_{3}$ are all intrinsic location functionals, and hence they all have their own partially ordered random set representations, denoted as $(S_1(g), \preceq_1)$, $(S_2(g), \preceq_2)$ and ($S_3(g), \preceq_3$). For positive sample paths, $L$ has $(S_1, \preceq_1)$ as its partially ordered random representation; otherwise for sample paths reaching level $-1$, $L$ has $(S_2, \preceq_2)$; otherwise, $L$ has $(S_3, \preceq_3)$. Combining the three cases gives a complete partially ordered random set representation for $L$. Thus, $L$ is an intrinsic location functional.

Next, we need to show that for any $F\in E_T$, there exists a periodic ergodic process with period $1$ such that $F$ is the distribution of $L$ over $[0,T]$ for such process. For any $F\in E_T$, let $f$ be its density function on $(0,T)$. We discuss two possible scenarios depending on whether $f(t)\geq 1$ for all $t$ or not.
\begin{enumerate}
\item If $f(t)\geq 1$ for all $t\in (0,T)$, we are going to show that there exists a periodic ergodic process with period $1$ and positive sample paths, such that $F$ is the distribution of $L_{1}$ on $[0,T]$ for that process. Since $L_1$ is a modified version of the location of the path supremum,   this part of the proof is postponed and will be resumed right after the proof of Theorem \ref{inpro}, in which we focus on the distribution of the location of the path supremum.
\item Otherwise, $f(t)=0$ for some $t$. Recall from the definition of $E_T$ that if $f(0+)\geq 1$ and $f(T-)\geq 1$, then $f(t)\geq 1$ for all $t\in(0,T)$. Hence in this case we must have $f(0+)=0$ or $f(T-)=0$. Assume $f(T-)=0$ for example. Take $u:=\inf\{t\in(0,T): f(t)=0\}$ and a sequence $\{t_n\in(u,T)\}_{n\in\mathbb N}$ such that $t_n\uparrow T$ as $n\to\infty$ and $f(t_n)=0$ for all $n$. The variation constraint applied to the intervals $(0,u)$ and $(u, t_n)$ implies that $f$ is non-increasing in $(0,u)$ and that $f(t)=0$ for $f\in[u,T)$, respectively. Symmetric results hold for the case where $f(0+)=0$. To summarize, if $f$ is the density function for a distribution in $E_T$ and $f(t)=0$ for some $t$, we have
\begin{enumerate}
\item[(1)] $f$ takes values in non-negative integers;
\item[(2)] Either there exists $u\in (0,T)$ such that $f$ is a non-increasing function in the interval $(0,u)$ and $f(t)=0$ for $t\in [u,T)$, or there exists $v\in (0,T)$ such that $f$ is a non-decreasing function in the interval $[v,T)$ and $f(t)=0$ for $t\in (0,v)$.
\end{enumerate}
\end{enumerate}
By symmetry, we only prove the case where $f$ is non-increasing in the interval $(0,u)$ and $f(t)=0$ for $t\in[u,T)$. Since the intrinsic location functional that we are going to use in this case, $L_2$, is a first hitting time, this part of the proof is postponed and will be resumed right after the proof of Proposition \ref{thmin}, which deals with this type of intrinsic location functionals.
\end{proof}

\begin{remark}
The proof of Theorem \ref{chage} actually implies a stronger result: all the distributions in $C(E_T)$ can be generated by a single intrinsic location functional, which is the location $L$ defined in the proof of the theorem.
\end{remark}

\begin{remark}
Among the three conditions defining the set $E_T$, the condition $(TV)$ is stable under convex combination, while the other two, integer values and a lower bound at level 1 under some conditions, are not. Therefore when passing from ergodic processes to stationary processes, these two conditions will not persist. However, this does not mean that they will simply disappear. They still affect the structure of the set of all possible distributions $I_T=C(E_T)$, but in a complicated way. While an explicit, analytical description of $I_T$ is not known, we point out in the following example that $I_T$ is indeed a proper subset of the set of all distributions solely satisfying condition $(TV)$.

Denote by $A_T$ the class of probability distributions on $[0,T]\cup\{\infty\}$ with densities satisfying the variation constraint $(TV)$. Let $T=1$ and consider a probability distribution $F$ with density function
\begin{align*}
f(t)=\begin{cases}
\frac{4}{3}, &\quad t\in(0,\frac{3}{4}),\\
0, &\quad t\in[\frac{3}{4},1).
\end{cases}
\end{align*}
From the construction of $f$, it is easy to check that $F\in A_T$. Suppose $F$ is also in the set $I_T$, then it can be written as an integral of the elements in the set $E_T$ with respect to a probability measure on $E_T$, as discussed in the proof of Theorem \ref{convexh}. Since $f(t)=0$ for all $t\in [\frac{3}{4},1)$, the variation constraint implies that any candidate density $g$ to construct $f$ must be non-increasing on the interval $(0,\frac{3}{4})$ and $g(t)=0$ for all $t\in[\frac{3}{4},1)$. Moreover, $g$ takes integer values, so there exists $g$ such that $g(t)=2$ for $t\in (0,\frac{3}{4})$. However, the integral of $g$ is
\begin{align*}
\int^T_{0}g(t)\mathrm{d}t=\frac{3}{2}>1,
\end{align*}
which means that there does not exist a distribution in $E_T$ such that $g$ is its density function. Therefore, $F\notin C(E_T)$, hence $I_T$ is a proper subset of $A_T$.
\end{remark}

\section{\textbf{Invariant intrinsic location functionals}}\label{sec:4}
In this section, we consider a special type of intrinsic location functionals, referred to as the \emph{invariant intrinsic location functionals}.
\begin{definition}
An intrinsic location functional $L$ is called \emph{invariant}, if it satisfies
\begin{enumerate}
\item $L(g, I)\neq\infty$ for any compact interval $I$ and $g\in H$.
\item $L(g, [0,1])=L(g, [a,a+1]) \mod 1$, for any $a\in \mathbb{R}$ and $g\in H$.
\end{enumerate}
\end{definition}

\begin{remark}
Invariance is a natural requirement for an intrinsic location functional on $S_1$. The projection of an interval with length of $1$ in $S_1$ forms a loop, with the starting and ending point being mapped to the same point. The above definition then requires that the location over the whole circle is always well-defined, and does not depend on the location of the starting/ending point.
\end{remark}

\begin{exmp}\label{exmp}
It is easy to see that the location of the path supremum
\begin{align*}
\tau_{g,[a,b]}=\inf \Bigl\{t\in [a,b]: g(t)=\sup_{a\leq s\leq b}g(s)\Bigr\}
\end{align*}
is an invariant intrinsic location functional, provided that the path supremum is uniquely achieved.
\end{exmp}

Besides the location of the path supremum, other invariant intrinsic location functionals include the location of the point with the largest/smallest slope (if the sample paths are in $C^1$), the location of the point with the largest/smallest curvature (if the sample paths are in $C^2$), \textit{etc}, provided the uniqueness of these locations. The related criteria for uniqueness often go back to checking the uniqueness of the path supremum/infimum in one period. Indeed, if the a periodic stationary process has sample paths in $C^1$ (resp. $C^2$), then its first (resp. second) derivative is again a periodic stationary process. For a Gaussian process $\mathbf X$, its derivative $\mathbf X'$ is still Gaussian, and Kim and Pollard \cite{kim1990} showed that the supremum is almost surely achieved at a unique point if $\text{Var}(X'(s), X'(t))\neq 0$ for $s\neq t$. In our periodic case, this means that the process has no period smaller than 1. Another condition was developed by Pimentel \cite{pimentel2014} for general processes with continuous sample paths.

For an invariant intrinsic location functional, we have the following lower bound for its density function.

\begin{proposition}\label{ILP}
For $T\in (0,1]$, any invariant intrinsic location functional $L$ and any periodic stationary process $\mathbf X$ with period $1$, the density $f^\mathbf {X}_{L,T}$ of $L$ on $(0,T)$ satisfies
\begin{align}
f^\mathbf{X}_{L,T}(t) \geq 1~~ \mbox{for all}~t\in (0,T).
\end{align}
\end{proposition}

\begin{proof}
Let $0<a<b<1$. Since $\mathbf{X}$ is stationary, we have
\begin{align}
\mathbb{P}(L(\mathbf{X},[0,1])\in (0,b-a))&=\mathbb{P}(L(\mathbf{X},[a, a+1])\in (a,b)).
\end{align}
By the assumption of invariant intrinsic location functionals, for any $a\in \mathbb{R}$,
\begin{align*}
L(\mathbf{X},[0,1])=L(\mathbf{X},[a,a+1])\mod 1.
\end{align*}
Then
\begin{align*}
\mathbb{P}(L(\mathbf{X},[0,1])\in (0,b-a))&=\mathbb{P}(L(\mathbf{X},[a,a+1])\in (a,b))\\
&=\mathbb{P}(L(\mathbf{X},[0,1])\in (a,b)).
\end{align*}
It means that $L(\mathbf{X},[0,1])$ follows a uniform distribution on the interval $[0,1]$. Thus, for any $t\in (0,1)$,
\begin{align*}
f^\mathbf{X}_{L,[0,1]}(t)=1.
\end{align*}
For any Borel set $B\in \mathcal B([0,T])$, $T\leq 1$, by condition 4 (stability under restrictions) in Definition \ref{def:ilf},
\begin{align*}
F^\mathbf{X}_{L,[0,T]}(B)\geq F^\mathbf{X}_{L,[0,1]}(B).
\end{align*}
Therefore, for any $0<t<T$,
\begin{align*}
f^\mathbf{X}_{L,T}(t)\geq f^\mathbf{X}_{L,1}(t)=1.
\end{align*}
\end{proof}

For a given invariant intrinsic location functional $L$ and $T\leq 1$, let $I^1_{L,T}$ be the collection of probability distributions of $L$ on $[0,T]$ for periodic stationary processes with period $1$. Let $E^1_{T}$ be the collection of probability distributions with no point mass at $\infty$, and (c\`{a}dl\`{a}g) densities $f$ on $(0,T)$ satisfying:
\begin{enumerate}
\item  $f$ takes values in positive integers for all $t\in(0,T)$;
\item  $f$ satisfies the condition $(TV)$.
\end{enumerate}
Then we have the following result regarding the structure of the set $I^1_{L,T}$, parallel to the result for general intrinsic location functionals, Theorem \ref{convexh}.

\begin{corollary}\label{inthm}
$I^1_{L,T}$ is a convex subset of $\mathcal P_T$. Moreover, $I^1_{L,T}\subseteq C(E^1_{T})$.
\end{corollary}

\begin{proof}
By Proposition \ref{ILP}, the density $f$ for any periodic ergodic process $\mathbf{X}$ with period $1$ satisfies $f(t)\geq 1$ for all $t\in(0,T)$. The rest of the proof follows in the same way as that of Theorem \ref{convexh}.
\end{proof}

Before proceeding to the next result, Theorem \ref{inpro}, which gives the other direction of the relation between $C(E_T^1)$ and the set of all possible distributions, we note that the definition of the location of the path supremum can be extended to the processes with c\`{a}dl\`{a}g sample paths. This extension will be helpful in the proof of Theorem \ref{inpro}.

\begin{remark}
For any periodic stationary process $\mathbf{X}$ with period $1$ and c\`{a}dl\`{a}g sample paths, let $X^\prime(t)= \lim \sup_{s\to t}X(s)$, $t\in \mathbb R$. Then $\mathbf{X}^\prime=\{X^\prime(t),~t\in \mathbb R\}$ has upper semi-continuous sample paths and its supremum over the interval can be attained. As a result, for any $\mathbf{X}$ with c\`{a}dl\`{a}g sample paths, the location of the path supremum for $\mathbf{X}$ can be defined as
\begin{align*}
\tau_{\mathbf{X}, T}:=\inf \left \{ t\in [0,T]: X^\prime(t)=\sup_{s\in [0,T]} X^\prime(s)\right \}.
\end{align*}
\end{remark}

Denote by $\mathcal L_I$ the set of invariant intrinsic location functionals. Let $I_T^1=\bigcup_{L\in\mathcal L_I}I^1_{L,T}$ be the collection of all the possible distributions for invariant intrinsic location functionals and periodic stationary processes with period 1 on $[0,T]$. The next result, in combination with Corollary \ref{inthm}, shows that $I^1_T=C(E^1_T)$.

\begin{theorem}\label{inpro}
For any $F\in C(E^1_T)$, there exists an invariant intrinsic location functional and a periodic stationary process with period $1$, such that $F$ is the distribution of this invariant intrinsic location functional for such process.
\end{theorem}

\begin{proof}
It suffices to show that for any distribution $F\in E^1_T$, there exists a periodic ergodic process $\mathbf{Y}$ with period $1$ such that $F$ is the distribution of the unique location of the path supremum for $\mathbf{Y}$ on $[0,T]$. By Proposition \ref{extreme}, the density function of $F$, denoted by $f$, takes non-negative integer values and satisfies the condition (TV). As a result, $f$ must be a piecewise constant function and has a unique decomposition
\begin{align}
f(t)=\sum^{m}_{i=1}\mathbb{I}_{(u_i,v_i]}(t),
\end{align}
where $m$ can be infinity and the intervals are maximal, in the sense that for any $i,j=1,\dots,m$, $(u_i,v_i]$ and $(u_j,v_j]$ have only three possible relations:
\begin{align*}
(u_i,v_i]\subset (u_j,v_j], \quad \text{or}\quad (u_j,v_j]\subset (u_i,v_i],\quad \text{or}\quad [u_i,v_i]\cap [u_j,v_j]=\emptyset.
\end{align*}
According to whether $u_i=0$ or $v_i=T$, we call the intervals of the form $(0,T]$, $(0,v_i]$, $(u_i,T]$ and $(u_i,v_i]$ the base, left, right and central block(s), respectively. Observe that property $1$ and $2$ in the defintion of $E_T^1$ are equivalent to requiring that there is at least one base block, and the number of the central blocks does not exceed the number of the base blocks.

We construct the stationary process in spirit of Proposition \ref{cha}. That is, first construct a periodic deterministic function $g$, and then uniformly shift its starting point to get $Y(t)=g(t+U)$, where $U$ is a uniform random variable on $[0,1]$. Let $m_1$ be the number of the base blocks in the collection. We group the entire collection of blocks into $m_1$ components by assigning to each base block at most one central block, and assigning the left and the right blocks in an arbitrary way. Assume $a=F(0)>0$ and $b=1-F(T)>0$. Let
\begin{align*}
d_1=\frac{1}{m_1}a~\text{and}~d_2=\frac{1}{m_1}b.
\end{align*}
For $j=1,\dots,m_1$, let
\begin{align*}
L_j&=d_1+\mbox{the total length of the blocks in the}~ j\mbox{th component}+d_2,
\end{align*}
then $\sum^{m_1}_{i=1}L_i=1$. Set $g(0) = 2$  and $g(L_1)=2$. Using the blocks of the first component, we will define the function $g$ on the interval $(0,L_1]$. If the first component has $l$ left blocks, $r$ right blocks and a central block, where $l$ and $r$ can potentially be infinity, we denote them by $(0, v_j]$, $j = 1,\dots, l$, $(u_k, T]$, $k = 1,\dots,r$ and $(u, v]$ respectively. The case where a central block does not exist corresponds to letting $u=v$. Set
\begin{align}
g\left(\sum^{j-1}_{i=1}v_i+\sum^{j}_{i=1}\frac{1}{2^{i+1}}d_1\right)=g\left(\sum^{j}_{i=1}v_i+\sum^{j}_{i=1}\frac{1}{2^{i+1}}d_1\right)=1+2^{-j},~j=1,\dots,l,
\end{align}
\begin{align}
\nonumber
g\left(d_1+\sum^{l}_{i=1}v_i\right)=g\left(d_1+\sum^{l}_{i=1}v_i+v\right)=g\left(d_1+\sum^{l}_{i=1}v_i+v+T-u\right)=\frac{1}{2},
\end{align}
and
\begin{align}
\nonumber
g\left(L_1-\sum^{j}_{i=1}\frac{1}{2^{i+1}}d_2-\sum^{j-1}_{i=1}(T-u_i)\right)&=g\left(L_1-\sum^{j}_{i=1}\frac{1}{2^{i+1}}d_2-\sum^{j}_{i=1}(T-u_i)\right)\\
&= 1+2^{-j},~ j = 1,\dots, r.
\end{align}
Next, if the values of $g$ at two adjacent points constructed above, $t_1<t_2$, are equal, we join them by a V-shaped curve satisfying some Lipschitz condition. We complete the function $g$ by filling in the other gaps with straight lines between adjacent points (with different values). With the similar construction, we can also define $g$ on the interval $[L_i,L_{i+1}]$, for $i=1,\dots,m_1-1$. Then $g$ is well defined on the interval $[0,1]$ and we extend $g$ as a periodic function with period $1$. If $a$ or $b$ equals to $0$, we take (the c\`{a}dl\`{a}g verion of) the limit of the corresponding construction with $a\downarrow 0$ or $b \downarrow 0$. We have a periodic ergodic process $Y$ as $Y(t) = g(t +U)$ for $t\in \mathbb R$, where $U$ is uniformly distributed on $[0,1]$. It is straightforward, though lengthy, by tracking the value of $L(g(t+U),[0,T])$ as a function of $U$, to see that the distribution of the location of the path supremum for $\mathbf{Y}$ is $F$. The proof is finally complete with an application of ergodic decomposition.
\end{proof}

\begin{remark}
Since the only random location used in the proof of Theorem \ref{inpro} is the location of the path supremum, we actually showed that the set of all possible distributions for invariant intrinsic location functionals is contained in the set of possible distributions solely for the location of path supremum. In this sense, the location of path supremum is a representative of the invariant intrinsic location functionals. This fact is related to the partially ordered random set representation of the intrinsic location functionals.
\end{remark}

\begin{remark}
In the part of introduction we mentioned the question as whether every relatively stationary process defined on an interval $[0,T]$ can always be extended to a periodic stationary process with a given period $T'>T$. Proposition \ref{ILP}, together with Theorem \ref{inpro}, gives a negative answer to this question. To see this, let $T'=1$, and consider the location of the path supremum denoted as $\tau$. Let $T''>1$. As a result of Theorem \ref{inpro}, a simple scaling shows that for a probability distribution $F$ on $[0,T]$ with its density function $f$ on $(0,T)$, as long as $f$ only takes values in positive multiples of $\frac{1}{T''}$ and satisfies the variation constraint $(TV)$, there exists a periodic ergodic process $\mathbf X$ with period $T''$, such that $F$ is the distribution of $\tau$ over the interval $[0,T]$ for $\mathbf X$. In particular, the value of $f(t)$ can be as small as $\frac{1}{T''}$ for some $t\in(0,T)$. Consider $\mathbf X|_{[0,T]}$, the restriction of $\mathbf X$ on $[0,T]$. It is a relatively stationary process. Suppose it can be extended to a periodic stationary process with period 1, denoted by $\mathbf Y$. Then by Proposition \ref{ILP}, the density of $\tau$ on $(0,T)$ for $Y$ is bounded from below by 1. Since $\mathbf Y$ agrees with $\mathbf X|_{[0,T]}$ on $[0,T]$, the lower bound 1 is also valid for $\mathbf X|_{[0,T]}$, hence $\mathbf X$ as well. This contradicts the fact that $f(t)$ can take value $\frac{1}{T''}$. We therefore conclude that the relatively stationary process $\mathbf X|_{[0,T]}$ does not have a stationary extension with period 1.
\end{remark}

We now turn back to the second part of the proof of Theorem \ref{chage} which we promised in the previous section.
\begin{proof}[Proof of Theorem \ref{chage}, Part II]
Recall that an intrinsic location functional $L_1$ is defined as follows:
$$
L_{1}(g,I)=\inf \left \{ t\in I: g(t)=\sup_{s\in I} g(s), ~g(t)\geq \frac{1}{2}\right \},
$$
and our goal in this part is to show that for any probability distribution $F\in E_T$ such that $f(t)\geq 1$ for all $t\in(0,T)$, there exists a periodic ergodic process with period 1 and non-negative sample paths, such that $F$ is the distribution of $L_1$ on $[0,T]$ for that process.

Comparing the conditions for the distribution $F$ and those for the distributions that we constructed in Theorem \ref{inpro}, the only difference is that $F$ allows a possible point mass at $\infty$ while the distributions in Theorem \ref{inpro} do not, because the location of the path supremum will always exist for processes with upper semi-continuous paths. This is the reason for which a modification is necessary. The way to construct the process changes accordingly, but not much. More precisely, let $F$ be our target distribution, with possible point masses $a$ and $b$ at the two boundaries 0 and $T$, respectively. Additionally, it has a possible point mass $c$ at $\infty$. Since the case where $c=0$ has been covered in the proof of Theorem \ref{inpro}, here we focus on $c>0$. Note that since $f-1$ also satisfies the variation constraint in this case, there exists at least one component which does not have a central block. Set this component as the first component. The construction of the process $X(t)=g(t+U)$, hence the function $g$, goes exactly in the same way as in the proof of Theorem \ref{inpro}, except for that now for this first component, instead of building the central block by setting
$$
g\left(d_1+\sum^{l}_{i=1}v_i\right)=g\left(d_1+\sum^{l}_{i=1}v_i+v\right)=g\left(d_1+\sum^{l}_{i=1}v_i+v+T-u\right)=\frac{1}{2},
$$
we set
$$
g\left(d_1+\sum_{i=1}^lv_i\right)=g\left(d_1+\sum_{i=1}^lv_i+T+c\right)=\frac{1}{2},
$$
and join them using a V-shaped curve as in the other cases. The construction of the rest of this component are shifted correspondingly. It is not difficult to verify that this part will contribute the desired mass at $\infty$.
\end{proof}

The variation constraint (TV) implies an upper bound for the density for intrinsic location functionals and stationary processes:
\begin{equation}\label{bounds}
f^\mathbf {X}_{L, T}(t)\leq \max \left(\frac{1}{t}, \frac{1}{T-t}\right), \quad 0<t<T.
\end{equation}
Moreover, such an upper bound was proved to be optimal \citep{samorodnitsky2013location}. With periodicity and the invariance property, we can now improve the above bound, and show that the improved upper bound is also optimal.

\begin{proposition}
Let $L$ be an invariant intrinsic location functional, $\mathbf{X}$ be a periodic stationary process with period $1$, and $T\in (0,1]$. Then the density $f^{\mathbf{X}}_{L,T}$ satisfies
\begin{align}\label{upb2}
f^{\mathbf{X}}_{L,T}(t) \leq \max \left(\lfloor \frac{1-T}{t}\rfloor, \lfloor \frac{1-T}{T-t}\rfloor \right)+2.
\end{align}
Moreover, for any $t\in (0,\frac{T}{2})$ such that $\frac{1-T}{t}$ is not an integer and $t\in[\frac{T}{2},T)$ such that $\frac{1-T}{T-t}$ is not an integer, there exists an invariant intrinsic location functional $L$ and a periodic stationary process $\mathbf X$ with period 1, such that the equality in (\ref{upb2}) is achieved at $t$.
\end{proposition}
\begin{proof}
Let $g^{\mathbf{X}}_{L,T}(t)=f^{\mathbf{X}}_{L,T}(t)-1$, then for every $0<t_1<t_2<T$, the variation constraint will be
\begin{align*}
\mathrm{TV}_{(t_1,t_2)}(g^{\mathbf{X}}_{L,T})&=\mathrm{TV}_{(t_1, t_2)}(f^{\mathbf{X}}_{L,T})
\leq f^{\mathbf{X}}_{L,T}(t_1)+f^{\mathbf{X}}_{L,T}(t_2)=g^{\mathbf{X}}_{L,T}(t_1)+g^{\mathbf{X}}_{L,T}(t_2)+2.
\end{align*}
Denote $a=\inf_{0<s\leq t}g^{\mathbf{X}}_{L,T}(s)$, $b=\inf_{t\leq s< T}g^{\mathbf{X}}_{L,T}(s)$. For any given $\epsilon>0$, there exists $u\in (0,t]$ such that
\begin{align*}
g^{\mathbf{X}}_{L,T}(u)\leq a+\epsilon,
\end{align*}
and there exists $v\in [t, T)$ such that
\begin{align*}
g^{\mathbf{X}}_{L,T}(v)\leq b+\epsilon.
\end{align*}
Note that
\begin{align}\label{eq:con3}
at+b(T-t)\leq \int^T_0 g^{\mathbf{X}}_{L,T}(s)\mathrm{d}s= \int^T_0 (f^{\mathbf{X}}_{L,T}(s)-1)\mathrm{d}s\leq 1-T.
\end{align}
Now appying the variation constraint to the interval $[u,v]$, we have
\begin{align*}
a+b+2\epsilon&\geq g^{\mathbf{X}}_{L,T}(u)+g^{\mathbf{X}}_{L,T}(v)\\
&\geq |g^{\mathbf{X}}_{L,T}(t)-g^{\mathbf{X}}_{L,T}(u)|+|g^{\mathbf{X}}_{L,T}(v)-g^{\mathbf{X}}_{L,T}(t)|-2\\
&\geq (g^{\mathbf{X}}_{L,T}(t)-a-\epsilon)_{+} +(g^{\mathbf{X}}_{L,T}(t)-b-\epsilon)_{+} -2.
\end{align*}
By the definition of $a$ and $b$, $a\leq g^{\mathbf{X}}_{L,T}(t)$ and $b\leq g^{\mathbf{X}}_{L,T}(t)$. Letting $\epsilon \to 0$, we have
\begin{align}\label{eq:con4}
g^{\mathbf{X}}_{L,T}(t)\leq a+b+1.
\end{align}
Combining (\ref{eq:con3}) and (\ref{eq:con4}) leads to
\begin{align*}
g^{\mathbf{X}}_{L,T}(t)\leq \max \left(\frac{1-T}{t}, \frac{1-T}{T-t}\right) + 1.
\end{align*}
Then for every $0<t<T$, an upper bound of $f^{\mathbf{X}}_{L,T}(t)$ is
\begin{align*}
f^{\mathbf{X}}_{L,T}(t)\leq \max \left(\frac{1-T}{t},\frac{1-T}{T-t} \right)+2.
\end{align*}
By Proposition \ref{extreme}, $f^{\mathbf{Y}}_{L,T}$ takes integer values for any periodic ergodic process $\mathbf Y$ with period $1$. Through ergodic decomposition, we further have the upper bound:
\begin{align*}
f^{\mathbf{X}}_{L,T}(t)\leq \max \left(\lfloor \frac{1-T}{t}\rfloor,\lfloor\frac{1-T}{T-t}\rfloor \right)+2.
\end{align*}

It remains to prove that such upper bound can be approached. For any $t\in (0,\frac{T}{2})$ such that $\frac{1-T}{t}$ is not an integer, define $f$ by
\begin{align*}
f(s)=\begin{cases}
1+\lfloor \frac{1-T}{t}\rfloor,~~&s\in (0,t),\\
2+\lfloor \frac{1-T}{t}\rfloor,~~&s\in [t, t+\varepsilon),\\
1,~~&s\in [t+\varepsilon,T),
\end{cases}
\end{align*}
where $\varepsilon$ is small enough so that $\int^T_0 f(s)\mathrm d s\leq 1$. As $f$ takes integer values and satisfies the condition $(TV)$, by Theorem \ref{inpro}, there exists an invariant intrinsic location functional $L$ and a periodic ergodic stationary process with period $1$ such that $f$ is the density of $L$ for such process. By similar construction, we can also find an invariant intrinsic location functional $L$ and a periodic ergodic process with period $1$ such that the density of $L$ for such process approaches $\lfloor \frac{1-T}{T-t} \rfloor$+2 at point $t$ for $t\in[\frac{T}{2}, T)$ satisfying $\frac{1-T}{T-t}$ is not an integer.
\end{proof}

We end this section by comparing the upper bound (\ref{upb2}) with the result (\ref{bounds}) for general stationary processes. For $t\leq \frac{T}{2}$, the following inequality holds between these two bounds:
\begin{align*}
\max \left\{\lfloor \frac{1-T}{t}\rfloor,\lfloor \frac{1-T}{T-t}\rfloor\right\}+2 \leq \frac{1-T}{t}+2\leq \frac{1}{t}=\max \left\{\frac{1}{t},\frac{1}{T-t}\right\}.
\end{align*}
For $t\geq \frac{T}{2}$,
\begin{align*}
\max \left\{\lfloor \frac{1-T}{t}\rfloor,\lfloor \frac{1-T}{T-t}\rfloor\right\}+2 \leq \frac{1-T}{T-t}+2\leq \frac{1}{T-t}=\max \left\{\frac{1}{t},\frac{1}{T-t}\right\}.
\end{align*}
Therefore, the upper bound in (\ref{upb2}) is always sharper than that in (\ref{bounds}). The improvement is most significant when $T$ is close to 1 and $t$ is close to $0$ or $T$.

\section{\textbf{First-time intrinsic location functionals}}\label{sec:5}
In this section, we introduce another type of intrinsic location functionals called the \emph{first-time intrinsic location functionals} via the partially ordered random set representation.
\begin{definition}
An intrinsic location functioanal $L$ is called a \emph{first-time intrinsic location functional}, if it has a partially ordered random set representation $(S(\mathbf X),\preceq)$ such that for any $t_1,t_2\in S(\mathbf{X})$,  $t_1\leq t_2$ implies $t_2\preceq t_1$.
\end{definition}

It is easy to see that the notion of the first-time intrinsic location functionals is a generalization of the first hitting times. As its name suggests, it contains all the intrinsic location functionals which can be defined as ``the first time'' that some condition is met.

\begin{proposition}\label{prop1}
Let $\mathbf{X}$ be a periodic stationary process with period $1$, and $L$ be a first-time intrinsic location functional. Fix $T\in(0,1]$. Then the density of $L$ on $(0,T)$ for $\mathbf{X}$ is non-increasing.
\end{proposition}

\begin{proof}
By ergodic decomposition, it suffices to prove the result for periodic ergodic process $\mathbf{X}$ with period $1$ having the representation
$X(t)=g(t+U)$, where $U$ is a uniform random variable on $[0,1]$. Let $(S,\preceq)$ be a partially ordered random set representation for $L$. By a similar argument as the discussion below (\ref{char3}), we have for $t\in(0,T)$,
\begin{align*}
f(t)=\left| \{s\in S(g)\cap (0,1]: a_s\geq t,~b_s\geq T-t\}\right |,
\end{align*}
where $a_s=\sup \{\Delta s\in \mathbb{R}: r\preceq s~\mbox{for all}~ r\in (s-\Delta s,s)\cap S(g)  \}$, $b_s= \sup \{\Delta s\in \mathbb{R}: r\preceq s~\mbox{for all}~ r\in(s,s+\Delta s)\cap S(g) \}$. By the definition of first-time intrinsic location functionals and that of $b_s$, we have
\begin{align*}
b_s=\infty,~ ~\mbox{for any}~s\in S(g).
\end{align*}
Thus for $t_1\leq t_2$,
\begin{align*}
f(t_2)=\left| \{s\in S(g)\cap (0,1]: a_s\geq t_2\}\right|~\mbox{and}~f(t_1)=\left| \{s\in S(g)\cap (0,1]: a_s\geq t_1\}\right|.
\end{align*}
If there exists $s\in S(g)\cap (0,1]$ such that $a_s\geq t_2$, then $a_s\geq t_2 \geq t_1$, which means that $f(t_1)\geq f(t_2)$. As a result, $f$ is non-increasing on the interval $(0,T)$.
\end{proof}

For any first-time intrinsic location functional $L$ and $T\leq 1$, let $I^M_{L,T}$ be the collection of the probability distributions of $L$ on $[0,T]$ for all periodic stationary processes with period $1$. Denote by $E^M_{T}$ the subset of $E_T$ consisting of the distributions with non-increasing density functions on $(0,T)$ and no point mass at $T$. Then we have the following result of the structure of $I^M_{L,T}$, parallel to Section $4$.
\begin{proposition}\label{mextreme}
$I^M_{L,T}$ is a convex subset of $\mathcal P_T$ and $I^M_{L,T}\subseteq C(E^M_{T})$.
\end{proposition}

The proof of Proposition \ref{mextreme} follows in a similar way to that of Theorem \ref{convexh} and is omitted.

As in the previous cases, the other direction also holds.
\begin{proposition}\label{thmin}
For any $F\in C(E^M_T)$, there exists a first-time intrinsic location functional and a periodic stationary process with period $1$, such that $F$ is the distribution of this first-time intrinsic location functional for such process.
\end{proposition}

\begin{proof}
We can actually use a single first-time intrinsic location functional for the proof. For example, let $L(g,I)=L_2(g,I)=\inf\{t\in I: g(t)=-1\}$ as defined in the proof of Theorem \ref{chage}. By ergodic decomposition, it suffices to show the result for distributions in $E^M_T$. Let $F$ be a probability distribution in $E_T^M$. Equivalently, $F$ is a probability distribution supported on $[0,T]\cup\{\infty\}$, with a possible point mass $a$ at 0, a possible point mass at $\infty$, and a non-increasing density function $f$ which takes non-negative integer values. Our goal is to show that there exists a periodic ergodic process with period $1$ such that the distribution of the first time reaching level $-1$ between $0$ and $T$ for such process is $F$. For ease of exposition, assume the point masses at 0 and at $\infty$ are both positive. The degenerate cases can be handled in a similar way. Since $f$ is non-increasing on $(0,T)$ with non-negative integer values, it can be written as
\begin{align*}
f(t)=\sum^{\infty}_{i=0} \mathbb I_{(0,u_i)}(t),
\end{align*}
where $u_i\geq u_{i+1}$. Define $s_i=\sum_{k=1}^i u_k$, $i=1,2,...$ and $s_0=0$. Let
\begin{align*}
g(s_i)=-1,~\mbox{for}~i=0,1,\dots
\end{align*}

In addition to $s_0,s_1,\dots$, we set $g(t)=-1$ for $t\in[s_{\infty},s_\infty+a]$ and $g(1)=-1$. Note that since $\int_0^1 f(t)\mathrm{d} t\leq 1$, $0\leq s_\infty\leq s_\infty+a\leq 1$. Next we join the consecutive points $(s_i,-1)$ and $(s_{i+1},-1)$, $i=0,1,\dots$ using V-shaped curves satisfying some Lipschitz condition with, for example, Lipschitz constant $1$. Similarly, use a V-shaped curve to join $(s_\infty+a,-1)$ and $(1,-1)$. Therefore, we can construct a periodic deterministic function $g$ with period $1$, and the required periodic ergodic process can be written as $X(t)=g(t+U)$ for $t\in \mathbb R$, where $U$ follows a uniform distribution on $[0,1]$. It is then routine to check that the distribution of $L$ is exactly $F$ by expressing the value of $L$ as a function of $U$.
\end{proof}

We have now all the pieces to complete the proof of Theorem \ref{chage}.

\begin{proof}[Proof of Theorem \ref{chage}, Part III]
Let $F\in E_T$, and $f$ be its density function on $(0,T)$. Recall that our goal in this part is to show that if $f$ is non-increasing with $\sup\{t: f(t)>0\}<T$, then for the intrinsic location functional $L_2(g,I)=\inf\{t\in I: g(t)=-1\}$, there exists a periodic ergodic process $\mathbf X$, such that $F$ is the distribution of $L_2$ on $[0,T]$ for $\mathbf X$. Note that since $f(t)$ takes value 0 as $t$ approaches $T$, by the definition of $E_T$, $F$ do not have a point mass at $T$. As a result, $F\in E_T^M$. Thus, by the proof of Proposition \ref{thmin}, $F$ is the distribution of $L_2$ for some periodic ergodic process with period 1.
\end{proof}

Denote by $\mathcal L_M$ the set of first-time intrinsic location functionals. Let $I_T^M=\bigcup_{L\in\mathcal L_M}I^M_{L,T}$ be the collection of all the possible distributions for first-time intrinsic location functionals and periodic stationary processes with period 1 on $[0,T]$. Denote by $A^M_T$ the class of probability distribution on $(0,T)$ with the properties that the corresponding density is c\`{a}dl\`{a}g and non-increasing. We would like to give a verification whether a function in $A^M_T$ is also in $I^M_T$. The recently developed concept of joint mixability \citep{wang2013bounds} is helpful.

In the following part, for any set $A$ of distributions,  we write $f\in_{d} A$, if there exists $F\in A$ such that $f$ is the corresponding density part of $F$.

In the definition below, we slightly generalize the concept of joint mixability to the case of possibly countably many distributions.
In the following $N$ is either a positive integer or it is infinity. If $N=\infty$, we interpret any tuple $(x_1,\dots,x_N)$ as $(x_i,~i=1,2,\dots)$.  Joint mixability and intrinsic location functionals are connected in Proposition \ref{prop:jm} below.

\begin{definition}\citep{wang2013bounds}\label{de:jm} Suppose $N\in \mathbb N\cup\{\infty\}$.
A random vector $(X_1,\dots,X_N)$ is said to be a joint mix if
$\mathbb P(\sum_{i=1}^N X_i = C)=1$ for some $C\in \mathbb R$. An $N$-tuple of distributions $(F_1,\dots,F_N)$ is said to be jointly mixable if there exists a joint mix $\mathbf X=(X_1,\dots,X_N)$ such that $X_i\sim F_i$, $i=1,\dots,N$.
\end{definition}

\begin{proposition}\label{prop:jm}
For any $f\in_d A^M_T$, let $N=\lceil f(0+) \rceil$, and define the distribution functions
\begin{equation}\label{eq:jm1}
F_i:\mathbb R \to [0,1], ~~ x\mapsto \min\{(i-f(x)\mathbb I_{\{x<T\}})_+, 1\}\mathbb I_{\{x\geq 0\}},~~i=1,\dots,N.
\end{equation}
Then
$f \in_d I^M_T$ if there exists a random vector $\mathbf X=(X_1,\dots,X_N)$ such that $X_i\sim F_i$, $i=1,\dots,N$ and $\mathbb P(\sum_{i=1}^N X_i\leq 1)=1$.
In particular, $f \in_d I^M_T$ if $(F_1 ,\dots, F_N)$ is jointly mixable.
\end{proposition}

\begin{proof}
Suppose that there exists a random vector $\mathbf X=(X_1,\dots,X_N)$ such that $X_i\sim F_i$, $i=1,\dots,N$ and $\mathbb P(\sum_{i=1}^N X_i\leq 1)=1$. For $\mathbf x=(x_1,\dots,x_N)$ satisfying $\sum_{i=1}^N x_i\leq 1$,
define  $$f_{\mathbf x}: [0,T] \to \mathbb R_+,~~ y\mapsto \sum_{i=1}^N \mathbb I_{\{y\le x_i\}}.$$
Obviously $f_{\mathbf x}$ is a non-increasing function and we can check
$$\int_0^T f_{\mathbf x}(y) \mathrm{d}y= \sum_{i=1}^N\int_0^T  \mathbb I_{\{y\le x_i\}}\mathrm{d}y= \sum_{i=1}^N x_i \le 1.$$
Thus, $f_{\mathbf x}$ is a non-increasing function on $[0,T]$ taking values in $\mathbb N_0$,  $\int_0^T f_{\mathbf x}(y) \mathrm{d}y\leq 1$, and hence $f_{\mathbf x}\in_d E^M_T$.
Moreover, for $y \in [0,T]$,
 \begin{align*}\mathbb{E}[f_{\mathbf X}(y)]&= \mathbb{E}\left[\sum_{i=1}^N \mathbb I_{\{y\leq X_i\}}\right]\\& =\lfloor f(y) \rfloor + \mathbb{E}\left[\mathbb I_{\{y\leq X_{\lfloor f(y) \rfloor}\}}\right]
 =\lfloor f(y) \rfloor + (f(y) -\lfloor f(y) \rfloor )=f(y).\end{align*}
Therefore, we conclude that $f \in_d I^M_T$ since it is a convex combination of $f_{\mathbf x}$.

Now suppose that  $(F_1 ,\dots, F_N)$ is jointly mixable.  Then there exists a joint mix $\mathbf X=(X_1,\dots,X_N)$ such that $X_i\sim F_i$, $i=1,\dots,N$ and $\mathbb P(\sum_{i=1}^N X_i=C)=1$ for some $C\in \mathbb R$. It suffices to verify that $C\leq 1$, which follows from
\begin{align}
C= \sum_{i=1}^N \mathbb{E}[X_i]&= \sum_{i=1}^N \int_0^T (1-F_i(x))\mathrm{d} x \nonumber\\
&=  \sum_{i=1}^N \int_0^T \min\{(f(x)-i+1)_+, 1\} \mathrm{d} x=\int_0^T  f (x) \mathrm{d} x \leq 1. \label{eq:jm2}
\end{align}
This completes the proof.
\end{proof}
\begin{remark}
In this section, $N$ might be infinity. It can be easily checked that in the case of $N=\infty$, the limit $\sum_{i=1}^N X_i$ in the above proof is well-defined since $\sum_{i=1}^N \mathbb{E}[X_i]\le 1$ and $X_i\ge 0$, $i=1,\dots,N$.
\end{remark}

\begin{corollary}
For a given density function $f\in_d A^M_T$, if there exists a step function $g\in_d E^M_T$ such that
\begin{align*}
g(t)\geq f(t),~ \mbox{for all}~ t\in (0,T),
\end{align*}
then $f\in_d I^M_T$.
\end{corollary}
\begin{proof}
For any $f\in_d A^M_T$, take $N$ and $F_i$, $i=1,\dots, N$ as defined in Proposition \ref{prop:jm}. Let $\mathbf X=(X_1,\dots, X_N)$ be a random vector such that $X_i\sim F_i$, $i=1, \dots, N$. Then we have
\begin{align*}
\sum_{i=1}^N X_i\leq \sum^{N}_{i=1}f^{-1}(i-1)\leq \int^T_0 g(t)\mathrm{d}t\leq 1
\end{align*}
hold almost surely. Thus, $f\in_d I^M_T$ by Proposition \ref{prop:jm}.
\end{proof}

\begin{corollary}\label{coro:2}
Suppose that $f\in_d A^M_T$ is convex on $[0,T]$ and
\begin{equation} \label{eq:jm3}
\sum_{i=0}^N f^{-1}(i)\leq 1 +f^{-1}(1).
\end{equation}
Then $f\in_d I^M_T$.
\end{corollary}
\begin{proof}
Let $N=\lceil f(0+) \rceil$ and $F_i,~i=1,\dots, N$ be as in \eqref{eq:jm1}.
Denote by $\mu_i$ the mean  of $F_i$ for $i=1,\dots, N$. Apparently $F_i$ has a non-increasing density supported in $[f^{-1}(i),f^{-1}(i-1)]$ for each $i=1,\dots, N$. By the convexity of $f$, we have $$\sum_{i=1}^N f^{-1}(i)+\max\{f^{-1}(i-1)-f^{-1}(i):i=1,\dots,N\} =\sum_{i=0}^N f^{-1}(i) -f^{-1}(1)\leq 1.$$
Since each $F_i$ has non-increasing densities, conditions in Corollary 4.7 of \cite{jakobsons2015general} are satisfied, giving that there exists $\mathbf X=(X_1,\dots,X_N)$ such that $X_i\sim F_i$, $i=1,\dots,N$ and
$$
\mathrm{ess\mbox{-}sup}\left(\sum_{i=1}^N X_i\right)=\max\left\{\sum_{i=1}^N f^{-1}(i) + \max_{i=1,\dots,N}\{f^{-1}(i-1)-f^{-1}(i)\}, \sum_{i=1}^N\mu_i\right\}\leq 1.
$$
The corollary follows from Proposition \ref{prop:jm}.
\end{proof}

\begin{remark}
Formally,  Corollary 4.7 of \cite{jakobsons2015general} only gives,  for any $\epsilon>0$ and $N\in \mathbb N$, the existence of $\mathbf X=(X_1,\dots,X_N)$ such that
 $$
 \mathrm{ess\mbox{-}sup}\left(\sum_{i=1}^N X_i\right)<\max\left\{\sum_{i=1}^N f^{-1}(i) + \max_{i=1,\dots,N}\{f^{-1}(i-1)-f^{-1}(i)\}, \sum_{i=1}^N\mu_i\right\}+\epsilon.
 $$
 A standard compactness argument would justify the case $\epsilon=0$ and $N=\infty$.
    Corollary 4.7 of \cite{jakobsons2015general}   requires the joint mixability of non-increasing densities; see Theorem 3.2 of \cite{wang2016joint}. For  $f\in_d A^M_T$, there is generally no constraints (except for location constraints) on the distributions $F_1,\dots,F_N$.
It is a difficult task to analytically verify whether a given tuple of distributions is jointly mixable.
For some other known necessary and sufficient conditions for joint mixability, see \cite{wang2016joint}.
\end{remark}

\begin{corollary}
Suppose that $f\in_d A^M_T$ is linear on its essential support $[0, b]$ and $f(b)=0$. Then $f\in_d I^M_T$.
\end{corollary}
\begin{proof}
Obviously the slope of the linear function $f$ on its support is not zero.
\begin{enumerate}%[(a)]
\item $\int_0^T f(x) \mathrm{d} x=1$.
In this case, $f$ is convex on $[0,T]$. We only need to verify \eqref{eq:jm3} in Corollary \ref{coro:2}.  Since $T< 1$ and since  $f$ integrates to 1, we have $N\geq 3$.
Note that, from integration by parts and change of variables, $\int_0^N f^{-1}(t)\mathrm{d} t=\int_0^T f(x)\mathrm{d} x =1$.
It follows from the linearity of $f$ that
\begin{align*}
\sum_{i=0}^N f^{-1}(i)-f^{-1}(1)&=\sum_{i=3}^N f^{-1}(i)+f^{-1}(0)+f^{-1}(2)\\
&=\sum_{i=3}^N f^{-1}(i)+ \int_0^2 f^{-1}(t)\mathrm{d} t\\&\leq \int_2^N f^{-1}(t)\mathrm{d} t+\int_0^2 f^{-1}(t)\mathrm{d} t  = 1.
\end{align*}
The desired result follows from Corollary \ref{coro:2}.
\item $\int_0^T f(x)\mathrm{d} x<1$.  This case can be obtained   from a mixture of (a) and $g\in_d E^M_T$ where $g:[0,T]\to \{0\}$. \qedhere
\end{enumerate}
\end{proof}

When $\int_0^T f(x)\mathrm{d} x<1$, we obtain a sufficient condition for $f\in_d A_M^T$ to be $f\in_d I^M_T$  using Proposition \ref{prop:jm} together with a result in \cite{embrechts2015aggregation}.
\begin{corollary}
For any $f\in_d A^M_T$, let $N=\lceil f(0+) \rceil$. Then $f \in_d I^M_T$ if
$$\max_{i=1,\dots,N}\{f^{-1}(i-1)-f^{-1}(i)\} \le 1-\int_0^T f(x)\mathrm{d} x.$$
\end{corollary}
\begin{proof}
Let $F_i,~i=1,\dots, N$ be as in \eqref{eq:jm1}. Apparently $F_i$ is supported in $[f^{-1}(i),f^{-1}(i-1)]$ for each $i=1,\dots, N$. Denote $L=\max\{f^{-1}(i-1)-f^{-1}(i):i=1,\dots,N\}$.
From Corollary A.3 of \cite{embrechts2015aggregation}, there exists a random vector $\mathbf X=(X_1,\dots,X_N)$ such that $X_i\sim F_i$, $i=1,\dots,N$ and $$\mathbb{P}\left(\left|\sum_{i=1}^N X_i-\sum_{i=1}^N\mathbb{E}\left[ X_i\right]\right|\le L\right)=1.$$
From \eqref{eq:jm2}, we have   $\sum_{i=1}^N\mathbb{E}\left[ X_i\right]=\int_0^T f(x)\mathrm{d} x$ and therefore,
$$\mathbb{P}\left( \sum_{i=1}^N X_i \le 1\right)\geq \mathbb{P}\left( \sum_{i=1}^N X_i \leq  L+\int_0^T f(x) \mathrm{d} x\right)=1.$$
\end{proof}

\section*{Acknowledgements} The authors would like to thank two anonymous referees for their insightful comments which were very helpful in improving the quality and the presentation of the paper. Jie Shen acknowledges financial support from the China Scholarship Council. Yi Shen and Ruodu Wang acknowledge financial support from the Natural Sciences and Engineering Research Council of Canada (RGPIN-2014-04840 and RGPIN-435844-2013, respectively).

\bibliographystyle{apalike}
\bibliography{Bibtexp}

\end{document}